\documentclass[a4paper, 11pt, reqno]{amsart}

\usepackage{microtype}
\usepackage{fullpage}

\usepackage{amssymb}
\usepackage{mathtools}
\usepackage{enumitem}

\newcommand{\BB}{\mathbb{B}}
\newcommand{\CC}{\mathbb{C}}
\newcommand{\dd}{\,\mathrm{d}}

\newcommand{\ext}{\operatorname{ext}}

\newcommand{\NN}{\mathbb{N}}
\newcommand{\pd}{\partial}

\newcommand{\RR}{\mathbb{R}}
\renewcommand{\Re}{\operatorname{Re}}
\newcommand{\tr}{\operatorname{tr}}

\theoremstyle{plain}
\newtheorem{theorem}{Theorem}[section]

\theoremstyle{plain}
\newtheorem{lemma}{Lemma}[section]

\theoremstyle{remark}
\newtheorem{remark}{Remark}[section]

\usepackage[
  colorlinks = true,
  linkcolor = blue,
  citecolor = green,
  linktoc=page,
]{hyperref}

\usepackage{cleveref}
    \crefname{equation}{eq.}{eqs.}
    \Crefname{equation}{Eq.}{Eqs.}

\title{A characterization of the subspace of radially symmetric functions in Sobolev spaces}
\author{Matthias Ostermann}
\address{University of Vienna, Faculty of Mathematics, Oskar-Morgenstern-Platz 1, 1090 Vienna, Austria}
\email{matthias.ostermann@univie.ac.at}

\begin{document}

\begin{abstract}
In this paper, we show that any Sobolev norm of nonnegative integer order of radially symmetric functions is equivalent to a weighted Sobolev norm of their radial profile. This establishes in terms of weighted Sobolev spaces on an interval a complete characterization of radial Sobolev spaces, which was open until now. As an application, we give a description of Sobolev norms of corotational maps.
\end{abstract}

\maketitle

\section{Introduction}
For $p\geq 1$ and integers $k\geq 0$, the \emph{classical Sobolev norm} is defined by
\begin{equation*}
\| \varphi \|_{W^{k,p}(\BB^{d}_{r})} = \Big( \sum_{0\leq |\alpha| \leq k} \| \pd^\alpha \varphi \|_{L^{p}(\BB^{d}_{r})}^{p} \Big)^{\frac{1}{p}} \,, \qquad \varphi\in C^{\infty}(\overline{\BB^{d}_{r}}) \,,
\end{equation*}
on the space $C^{\infty}(\overline{\BB^{d}_{r}})$ of smooth functions on the open ball $\BB^{d}_{r} \subset \RR^{d}$ of radius $r>0$ centred around the origin, whose derivatives of all orders are continuous up to the boundary. The \emph{Sobolev space} $W^{k,p}(\BB^{d}_{r})$ is realized as the completion of the space $C^{\infty}(\overline{\BB^{d}_{r}})$ with respect to the norm $\| \,.\, \|_{W^{k,p}(\BB^{d}_{r})}$. We remark that the spaces obtained by completion agree with the usual Sobolev spaces defined in terms of weak derivatives, see \cite[p. 266, Theorem 3]{MR2597943}. On $\RR^{d}$, the Sobolev norm may be arranged in terms of \emph{classical homogeneous Sobolev norms}
\begin{equation*}
\| \varphi \|_{\dot{W}^{k,p}(\RR^{d})} = \Big( \sum_{|\alpha|=k} \| \pd^\alpha \varphi \|_{L^{p}(\RR^{d})}^{p} \Big)^{\frac{1}{p}} \,, \qquad \varphi\in C^{\infty}_{\mathrm{c}}(\RR^{d}) \,,
\end{equation*}
on the space $C^{\infty}_{\mathrm{c}}(\RR^{d})$ of smooth functions with compact support. The homogeneous Sobolev space $\dot{W}^{k,p}(\RR^{d})$ is the completion of the space $C^{\infty}_{\mathrm{c}}(\RR^{d})$ with respect to the norm $\| \,.\, \|_{\dot{W}^{k,p}(\RR^{d})}$. The underlying spatial domains are invariant under rotations and therefore, the radial subspaces
\begin{equation*}
W^{k,p}_{\mathrm{rad}}(\BB^{d}_{r}) \subset W^{k,p}(\BB^{d}_{r})
\qquad\text{and}\qquad
\dot{W}^{k,p}_{\mathrm{rad}}(\RR^{d}) \subset \dot{W}^{k,p}(\RR^{d}) \,,
\end{equation*}
which are the completion of the radial spaces $C^{\infty}_{\mathrm{rad}}(\overline{\BB^{d}_{r}})$ and $C^{\infty}_{\mathrm{rad}}(\RR^{d}) \cap C^{\infty}_{\mathrm{c}}(\RR^{d})$ endowed with the norms $\| \,.\, \|_{W^{k,p}(\BB^{d}_{r})}$ and $\| \,.\, \|_{\dot{W}^{k,p}(\RR^{d})}$, respectively, appear frequently. Here and in the following, we encounter the space of smooth radial functions
\begin{equation*}
C^{\infty}_{\mathrm{rad}}(\overline{\BB^{d}_{r}}) = \{ \varphi \in C^{\infty}(\overline{\BB^{d}_{r}}) \mid \varphi \text{ is radially symmetric} \}
\end{equation*}
and also the space of smooth even functions
\begin{equation*}
C^{\infty}_{\mathrm{ev}}([0,r]) = \{ f\in C^{\infty}([0,r]) \mid \forall\,n\in\NN_{0}: f^{(2n+1)}(0) = 0 \} \,.
\end{equation*}
We recall the following representations of smooth radial functions, for which we give a short proof in \Cref{RepresentativeEven,RepresentativeRadial}.
\begin{lemma}
\label{RadialRepresentationLemma}
Let $d\geq 2$ be an integer and $r>0$. Let $\varphi \in C^{\infty}_{\mathrm{rad}}(\overline{\BB^{d}_{r}})$. Then, there are $f_{\varphi} \in C^{\infty}_{\mathrm{ev}}([0,r])$ and $\widetilde{f}_{\varphi} \in C^{\infty}([0,r^{2}])$ such that
\begin{equation}
\label{RadialRepresentation}
\varphi(x) = f_{\varphi}(|x|)
\qquad \text{and} \qquad
\varphi(x) = \widetilde{f}_{\varphi}(|x|^{2}) \,.
\end{equation}
\end{lemma}
This raises the natural question of how Sobolev norms of radial functions are related to Sobolev norms of their radial profiles. In case $k=0$, the Lebesgue norm of a radial function is expressed as a weighted Lebesgue norm of its radial representatives \eqref{RadialRepresentation}, namely
\begin{equation}
\label{radialLebesgue}
\| \varphi \|_{L^{p}(\BB^{d}_{r})}^{p} = |S^{d-1}| \Big\| (\,.\,)^{\frac{d-1}{p}} f_{\varphi} \Big\|_{L^{p}(0,r)}^{p} = \frac{|S^{d-1}|}{2} \Big\| (\,.\,)^{\frac{d-2}{2p}} \widetilde{f}_{\varphi} \Big\|_{L^{p}(0,r^{2})}^{p} \,,
\end{equation}
where $|S^{d-1}|$ is the surface area of the $(d-1)$-dimensional unit sphere in $\RR^{d}$. Analogues of these formulas for $W^{1,p}$-norms are derived in the same way. In this context, the pioneering radial lemmas are the famous inequality due to W. Strauss \cite{MR454365} and extensions thereof by W. M. Ni \cite{MR674869} and H. Berestycki and P.-L. Lions \cite{LIONS1982315}, \cite{MR695535}. Strauss-type inequalities for some higher-order homogeneous Sobolev spaces and also nonradial versions were obtained by Y. Cho and T. Ozawa in \cite{MR2538202}. This interplay of decay and regularity has been further investigated by W. Sickel, L. Skrzypczak and J. Vybiral \cite{MR1790248}, \cite{MR2902295}, \cite{MR2921084}, \cite{MR3330617} for the more general classes of Besov spaces and Lizorkin-Triebel spaces. Among them, they also studied radial subspaces of Sobolev spaces. Their results yield a linear isomorphism for radial Sobolev spaces $W^{2m,p}_{\mathrm{rad}}(\RR^{d})$ of even order $k=2m$ for $d\geq 2$ and $p>1$ into the closure of $C^{\infty}_{\mathrm{c}}(\RR) \cap C^{\infty}_{\mathrm{ev}}(\RR)$ with respect to a weighted norm expressed in terms of the $m$-th power of the radial Laplace operator, which is given by
\begin{equation*}
(\Delta_{\mathrm{r}} f_{\varphi})(\rho) := \rho^{-(d-1)} \pd_{\rho} \big( \rho^{d-1} \pd_{\rho} f_{\varphi}(\rho) \big) \,,
\end{equation*}
see \cite[Theorem 7]{MR2902295}. However, they noticed in \cite[Remark 6]{MR2902295} that a complete characterization for radial Sobolev spaces of higher order is still open and expect that a more elementary description can be given. So far, the most comprehensive results to our knowledge are the radial lemmas in \cite{MR2838041}, where D. G. de Figueiredo, E. M. dos Santos and O. H. Miyagaki used that for all $j\in\NN$ and $\varphi(x) = f_{\varphi}(|x|)$
\begin{equation}
\label{RadialRepresentationDerivative}
f_{\varphi}^{(j)}(|x|) = \sum_{|\alpha|=j} (\pd^{\alpha} \varphi)(x) \frac{j!}{\alpha!} \frac{x^{\alpha}}{|x|^{j}}
\end{equation}
to prove an embedding of $W^{k,p}_{\mathrm{rad}}(\BB^{d}_{1})$ into a weighted Sobolev space on the unit interval, which is an isomorphism precisely if $0\leq k < d/2 + 1$. A similar characterization based on identity \eqref{RadialRepresentationDerivative} was given by I. Glogi\'{c} in \cite{MR4469070} in case $p=2$, who recently proved in \cite{2022arXiv220706952G} a Strauss-type inequality for partial derivatives of radial functions. Alternatively, general $L^{2}$-based Sobolev norms on $\RR^{d}$ can be expressed equivalently with the Fourier transform, see \cite[p. 297, Theorem 8]{MR2597943}. If additionally the space dimension is odd, this can be used along with properties of the radial Fourier transform and Bessel functions to characterize radial Sobolev norms as a one-dimensional Sobolev norm of a weighted radial representative, see \cite{MR4338226}, \cite{MR4661000}.
\par\medskip
In this paper, we provide in \Cref{THMnorm,THMspace} the missing characterization of radial Sobolev spaces. We also apply our results in \Cref{THMcorot} to describe Sobolev norms of corotational maps.
\subsection{Characterization of radial Sobolev norms}
It is not obvious how the formulas in \eqref{radialLebesgue} extend to higher radial Sobolev norms. Simply differentiating \eqref{RadialRepresentation} and expressing partial derivatives of a radial function in terms of derivatives of its radial representative produces via the chain rule terms with singularities at the origin. To overcome this spurious problem, we study in \Cref{DerRad} systematically derivatives of radially symmetric functions in terms of the operation
\begin{equation*}
(D f)(\rho) := \rho^{-1} f'(\rho) \,, \quad \rho > 0 \,,
\end{equation*}
for $f\in C^{\infty}_{\mathrm{ev}}([0,r])$. The operation $(Df)(\rho)$ can be smoothly extended to all $\rho\in[0,r]$ to give a derivation
\begin{equation*}
D: C^{\infty}_{\mathrm{ev}}([0,r]) \rightarrow C^{\infty}_{\mathrm{ev}}([0,r]) \,, \quad f \mapsto \Big( \rho \mapsto (Df)(\rho) \Big) \,.
\end{equation*}
With this, we come to the main results of this paper.
\begin{theorem}
\label{THMnorm}
Let $d\geq 2$, $k\geq 0$ be integers and $p\geq 1$, $r>0$. For any $\varphi \in C^{\infty}_{\mathrm{rad}}(\overline{\BB^{d}_{r}})$, let $f_{\varphi} \in C^{\infty}_{\mathrm{ev}}([0,r])$ and $\widetilde{f}_{\varphi} \in C^{\infty}([0,r^{2}])$ be the radial profiles of $\varphi$ such that
\begin{equation*}
\varphi(x) = f_{\varphi}(|x|) = \widetilde{f}_{\varphi}(|x|^{2}) \,,
\end{equation*}
see \Cref{RadialRepresentationLemma}.
\begin{enumerate}[itemsep=1em,topsep=1em]
\item Then, the equivalence of norms
\begin{equation*}
\| \varphi \|_{W^{k,p}(\BB^{d}_{r})} \simeq \sum_{j=0}^k \Big\| (\,.\,)^{\frac{d-1}{p}+j} D^{j} f_{\varphi} \Big\|_{L^{p}(0,r)} \simeq \sum_{j=0}^k \Big\| (\,.\,)^{\frac{d-2}{2p}+\frac{j}{2}} \widetilde{f}_{\varphi}^{(j)} \Big\|_{L^{p}(0,r^{2})}
\end{equation*}
holds for all $\varphi \in C^{\infty}_{\mathrm{rad}}(\overline{\BB^{d}_{r}})$.
\medskip
\item Similarly, the equivalence of homogeneous Sobolev norms
\begin{equation*}
\| \varphi \|_{\dot{W}^{k,p}(\RR^{d})} \simeq \Big\| (\,.\,)^{\frac{d-1}{p} + k} D^k f_{\varphi} \Big\|_{L^{p}(0,\infty)} \simeq \Big\| (\,.\,)^{\frac{d-2}{2p}+\frac{k}{2}} \widetilde{f}_{\varphi}^{(k)} \Big\|_{L^{p}(0,\infty)}
\end{equation*}
holds for all $\varphi\in C^{\infty}_{\mathrm{rad}}(\RR^{d}) \cap C^{\infty}_{\mathrm{c}}(\RR^{d})$.
\end{enumerate}
\end{theorem}
To describe the radial Sobolev spaces, let us consider the weighted norms and associated spaces as they appear in \Cref{THMnorm}. For $d \geq 2$, $k \geq 0$  integers and $p\geq 1$, $r>0$, let $w = \big( w_{j} \big)_{j=0}^k$ be the collection of weight functions that are given by $w_{j}(\rho) = \rho^{\frac{d-2+pj}{2}}$ and define a \emph{weighted Sobolev norm} by
\begin{equation}
\label{WeightedSobolevNorm}
\| f \|_{W^{k,p}\big((0,r^{2}),w\big)} = \Big( \sum_{j=0}^k \Big\| w_{j}^{\frac{1}{p}} f^{(j)} \Big\|_{L^{p}(0,r^{2})}^{p} \Big)^{\frac{1}{p}} \,, \qquad f\in C^{\infty}([0,r^{2}]) \,.
\end{equation}
The \emph{weighted Sobolev space} $W^{k,p}\big((0,r^{2}),w\big)$ is the completion of $C^{\infty}([0,r^{2}])$ with respect to this norm. Homogeneous versions $\dot{W}^{k,p}\big((0,\infty),(\,.\,)^{\frac{d-2+pk}{2}}\big)$ are defined analogously. Associated to the representation \eqref{RadialRepresentation}, we consider now the densely defined linear \emph{trace operator}
\begin{equation*}
\tr: C^{\infty}_{\mathrm{rad}}(\overline{\BB^{d}_{r}}) \subset W^{k,p}_{\mathrm{rad}}(\BB^{d}_{r}) \rightarrow W^{k,p}\big( (0,r^{2}),w \big) \,, \qquad \varphi \mapsto \widetilde{f}_{\varphi} \,,
\end{equation*}
and \emph{extension operator}
\begin{equation*}
\ext: C^{\infty}([0,r^{2}]) \subset
W^{k,p}\big( (0,r^{2}),w \big) \rightarrow W^{k,p}_{\mathrm{rad}}(\BB^{d}_{r}) \,, \qquad f \mapsto f(|\,.\,|^{2}) \,.
\end{equation*}
Their bounded linear extensions yield the desired characterization of the radial subspaces.
\begin{theorem}
\label{THMspace}
Let $d\geq 2$, $k\geq 0$ be integers and $p\geq 1$, $r>0$. The operator $\tr: W^{k,p}_{\mathrm{rad}}(\BB^{d}_{r}) \rightarrow W^{k,p}\big( (0,r^{2}),w \big)$ is a bounded linear isomorphism with bounded linear inverse operator $\ext: W^{k,p}\big( (0,r^{2}),w \big) \rightarrow W^{k,p}_{\mathrm{rad}}(\BB^{d}_{r})$. In the same way, $\dot{W}^{k,p}_{\mathrm{rad}}(\RR^{d})$ and $ \dot{W}^{k,p}\big( (0,\infty), (\,.\,)^{\frac{d-2+pk}{2}} \big)$ are isomorphic as normed linear spaces.
\end{theorem}
\subsection{Application to Sobolev norms of corotational maps}
We say that a map $F: \BB^{d}_{r} \rightarrow\mathbb{C}^{d}$, $F(x) = \big( F_{1}(x), \ldots, F_{d}(x) \big)$, is \emph{corotational} if there exists a function $f:[0,r)\rightarrow\mathbb{C}$ such that
\begin{equation*}
F_{i}(x) = x_{i} f(|x|)
\end{equation*}
for $i = 1,\ldots,d$, and call $f$ the radial profile of $F$. Sobolev norms of $L^{2}$-based Sobolev spaces $H^{k}(\BB^{d}_{r}) := W^{k,2}(\BB^{d}_{r})$ are defined for a corotational map $F \in C^{\infty}(\overline{\BB^{d}_{r}}, \CC^{d})$ as
\begin{equation*}
\| F \|_{H^{k}(\BB^{d}_{r})} = \Big( \sum_{i=1}^{d} \| F_{i} \|_{H^{k}(\BB^{d}_{r})}^{2} \Big)^{\frac{1}{2}} \,.
\end{equation*}
Corotational maps are a special case of so-called equivariant maps between rotationally symmetric manifolds, see e.g., \cite[p. 109, Definition 8.1]{MR1674843} for a general definition. As such, they appear as critical points to certain geometric action functionals, most prominently in the wave maps equation and in Yang-Mills equations. Equivalence results for Sobolev norms of corotational maps then allow to link the function spaces for the reduced equations of the radial profile to the spaces that underlie the general equations. Radial lemmas that characterize the Sobolev norm of $F$ in terms of its radial profile $f$ were formulated in \cite{MR1278351} on $\RR^{d}$ when $k = \frac{d}{2}$ in the context of wave maps and lately in \cite{MR4469070} with the restriction $0\leq k < \frac{d}{2} + 2$ in the context of Yang-Mills equations. However, it is expected in \cite[Remark A.4]{MR4469070} that an equivalence result also holds for spaces of higher order. Since a complete charactersiation of Sobolev norms for corotational maps is of current interest in the study of stability of blowup in nonlinear wave equations, we provide the full result in the following theorem.
\begin{theorem}
\label{THMcorot}
Let $d\geq 2$, $k\geq 0$ be integers and $p\geq 1$, $r>0$. We have
\begin{equation*}
\|F\|_{H^{k}(\BB^{d}_{r})} \simeq \| f(|\,.\,|) \|_{H^{k}(\BB^{d+2}_{r})}
\end{equation*}
for all corotational maps $F \in C^{\infty} ( \overline{\BB^{d}_{r}},\mathbb{C}^{d} )$ with radial profile $f\in C^{\infty}_{\mathrm{ev}} ( [0,r] )$.
\end{theorem}
\begin{remark}
The analogous statement of \Cref{THMcorot} for homogeneous Sobolev norms on $\RR^{d}$ has been proved in \cite{MR4469070} by employing Fourier techniques.
\end{remark}
\subsection{Notation}
The set of nonnegative integer numbers, real numbers and complex numbers is denoted by $\NN_{0}$, $\RR$ and $\CC$, respectively. In $d$-dimensional Euclidean space $\RR^{d}$, we denote the Euclidean length of $x = (x_{1},\ldots,x_{d})\in\RR^{d}$ by $|x| = \Big( \sum_{i=1}^{d} (x_{i})^{2} \Big)^{\frac{1}{2}}$. The open ball of radius $r>0$ centred around the origin is defined by $\BB^{d}_{r} = \{ x\in\RR^{d} \mid |x|<r \}$.
\par\medskip
First and second derivatives of a function $f$ in one real variable are denoted by $f'$ and $f''$, respectively, whereas the $j$-th derivative of $f$ is denoted by $f^{(j)}$. Given a function $\varphi$ in $d$ real variables, we denote its $j$-th partial derivative by $\pd_{j} \varphi$. As usual, for multi-indices $\alpha = (\alpha_{1},\ldots,\alpha_{d})\in\NN_{0}^{d}$, we employ for higher-order partial derivatives multi-index notation $\pd^{\alpha}\varphi = \pd_{1}^{\alpha_{1}}\ldots\pd_{d}^{\alpha_{d}}\varphi$. We define the length of a multi-index $\alpha\in\NN_{0}^{d}$ by $|\alpha| = \alpha_{1} + \ldots + \alpha_{d}$, the factorial by $\alpha! = \alpha_{1}! \cdot \ldots \cdot \alpha_{d}!$ and powers by $x^{\alpha} = x_{1}^{\alpha_{1}} \ldots x_{d}^{\alpha_{d}}$.
\par\medskip
When dealing with inequalities for nonnegative terms $a_{\lambda},b_{\lambda} \geq 0$ which depend on a variable $\lambda\in\Lambda$, we follow the usual convention and define the relation $a_{\lambda} \lesssim b_{\lambda}$ if there exists a constant $c>0$ such that $a_{\lambda} \leq c b_{\lambda}$ uniformly for all $\lambda\in\Lambda$. As customary, we define that $a_{\lambda} \gtrsim b_{\lambda}$ holds if $b_{\lambda} \lesssim a_{\lambda}$ holds, and $a_{\lambda}\simeq b_{\lambda}$ if both relations $a_{\lambda} \lesssim b_{\lambda}$ and $a_{\lambda} \gtrsim b_{\lambda}$ hold. This notation can be flexibly applied e.g., to uniform pointwise estimates for functions or to the equivalence between norms on function spaces.
\section{Smooth radial functions}
In the following, we elucidate some relations between a radially symmetric function and its radial profile. Recall that for $d\geq 2$, a function $\varphi: \BB^{d}_{r} \rightarrow \CC$ is \emph{radially symmetric} if $\varphi(Rx) = \varphi(x)$ for any rotation $R\in SO(d)$ and for all $x\in \BB^{d}_{r}$.
\subsection{Radial profiles}
We begin with an elementary fact which is attributed to H. Whitney \cite{MR7783}.
\begin{lemma}
\label{RepresentativeEven}
A function $f:[0,r]\rightarrow\CC$ belongs to $C^{\infty}_{\mathrm{ev}}([0,r])$ if and only if there is a function $\widetilde{f}\in C^{\infty}([0,r^{2}])$ such that
\begin{equation*}
f(\rho) = \widetilde{f}(\rho^{2})
\end{equation*}
for all $\rho\in[0,r]$.
\end{lemma}
\begin{proof}
``$\Rightarrow$'': Let $f\in C^{\infty}_{\mathrm{ev}}([0,r])$. Define $\widetilde{f}\in C^{\infty}((0,r^{2}])$ by $\widetilde{f}(\rho) = f(\sqrt{\rho})$. We claim that for all $n\geq 1$, the $n$-th deriviative of $\widetilde{f}$ is given by
\begin{equation}
\label{proofWhitney}
\widetilde{f}^{(n)}(\rho^{2}) = \frac{1}{2^{2n-1} (n-1)!} \int_{0}^1 (1-t^{2})^{n-1} f^{(2n)}(t\rho) \dd t \,, \quad \rho>0 \,.
\end{equation}
Indeed, differentiating the identity $\widetilde{f}(\rho^{2}) = f(\rho)$ and using the fundamental theorem of calculus yields
\begin{equation*}
\widetilde{f}'(\rho^{2}) = \frac{1}{2} \int_{0}^1 f''(t\rho) \dd t \,.
\end{equation*}
Assuming that the claim \eqref{proofWhitney} holds for an arbitrary but fixed $n\geq 1$, differentiating it and performing an integration by parts yields
\begin{align*}
2\rho \widetilde{f}^{(n+1)}(\rho^{2}) &= \frac{1}{2^{2n-1}(n-1)!} \int_{0}^1 (1-t^{2})^{n-1} t f^{(2n+1)}(t\rho) \dd t \\&=
- \frac{1}{2^{2n-1} (n-1)!} \int_{0}^1 \frac{1}{2n} \pd_t (1-t^{2})^{n} f^{(2n+1)}(t\rho) \dd t \\&=
\frac{\rho}{2^{2n} n!} \int_{0}^1 (1-t^{2})^{n} f^{(2(n+1))}(t\rho) \dd t
\end{align*}
and hence the claim is proved. Thus $\lim_{\rho\to 0}\widetilde{f}^{(n)}(\rho)$ exists for all $n\geq 1$, so $\widetilde{f}\in C^{\infty}([0,r^{2}])$.
\par\medskip
``$\Leftarrow$'': If $\widetilde{f}\in C^{\infty}([0,r^{2}])$, then the function $f$ defined by $f(\rho) = \widetilde{f}(\rho^{2})$ is smooth and even.
\end{proof}
This yields an important characterization of smooth radial functions.
\begin{lemma}
\label{RepresentativeRadial}
Let $d\geq 2$ be an integer and $r>0$. A function $\varphi: \BB^{d}_{r} \rightarrow \CC$ belongs to $C^{\infty}_{\mathrm{rad}}(\overline{\BB^{d}_{r}})$ if and only if there is a function $f_{\varphi}\in C^{\infty}_{\mathrm{ev}}([0,r])$ such that
\begin{equation*}
\varphi(x) = f_{\varphi}(|x|)
\end{equation*}
for all $x\in\BB^{d}_{r}$.
\end{lemma}
\begin{proof}
``$\Rightarrow$'': Assume $\varphi\in C^{\infty}_{\mathrm{rad}}(\overline{\BB^{d}_{r}})$. Define $f_{\varphi}\in C^{\infty}([0,r])$ by $f_{\varphi}(\rho) = \varphi(\rho e_{d})$, where $e_{d}\in\RR^{d}$ is the $d$-th standard basis vector. Then $f_{\varphi}$ is a smooth and even function such that $\varphi(x) = f_{\varphi}(|x|)$.
\par\medskip
``$\Leftarrow$'': Conversely, suppose $f\in C^{\infty}_{\mathrm{ev}}([0,r])$ and define $\varphi(x) = f(|x|)$. Then $\varphi$ is radial and smooth by \Cref{RepresentativeEven}.
\end{proof}
\begin{remark}
The analogous statements of \Cref{RepresentativeEven,RepresentativeRadial} clearly apply to functions $f:[0,\infty)\rightarrow\CC$ and $\varphi:\RR^{d}\rightarrow\CC$, respectively, too.
\end{remark}
\subsection{Derivatives of radial functions}
A key towards \Cref{THMnorm} is to expand partial derivatives of smooth radial functions in terms of the action of the derivation
\begin{equation*}
D: C^{\infty}_{\mathrm{ev}}([0,r]) \rightarrow C^{\infty}_{\mathrm{ev}}([0,r]) \,, \quad f \mapsto \Big( \rho \mapsto \rho^{-1} f'(\rho) \Big) \,,
\end{equation*}
on their radial representative, and vice versa. Incidentally, identities for partial derivatives of radial functions have been known already to R. Lyons and K. Zumbrun \cite{MR1227524}.
\begin{lemma}
\label{DerRad}
Let $d\geq 2$ be an integer and $r>0$. Let $\varphi \in C^{\infty}_{\mathrm{rad}}(\overline{\BB^{d}_{r}})$ with radial profile $f_{\varphi} \in C^{\infty}_{\mathrm{ev}}([0,r])$ according to \Cref{RepresentativeRadial}. Then, for any $n\in\mathbb{N}$ and any multi-index $\alpha\in\mathbb{N}_{0}^{d}$ of length $|\alpha|=n$ we have
\begin{equation}
\label{pdalphaphi}
\pd^{\alpha}_{x} \varphi(x) = \sum_{j=\left\lceil\frac{n}{2}\right\rceil}^{n} \frac{1}{2^{n-j} (n-j)!} \big( \Delta^{n-j}_{x} x^{\alpha} \big) \big( D^{j} f_{\varphi} \big)(|x|) \,,
\end{equation}
where $\Delta_{x}$ denotes the Laplace operator. Conversely, for any $n\in\mathbb{N}$ and multi-index $\alpha\in\mathbb{N}_{0}^{d}$ of length $|\alpha|=n$ there are homogeneous polynomial functions $q_{\alpha}$ of degree $2n$ such that
\begin{equation}
\label{Dnf}
|x|^{n} (D^{n} f_{\varphi})(|x|) = \sum_{|\alpha|=n} q_{\alpha}(\tfrac{x}{|x|}) \pd^{\alpha}_{x} \varphi(x) \,.
\end{equation}
\end{lemma}
\begin{proof}
We switch from multi-index notation to $d$-index notation. Namely, for $d,n\in\mathbb{N}$, a \emph{$d$-index of length $n$} is a tuple $I = (i_{1},\ldots,i_{n})$ of numbers $i_{1},\ldots,i_{n} \in \{1,\ldots,d\}$. The set of all $d$-indices of length $n$ is given by the Cartesian product $\mathcal{I}^{n}_{d} := \{1,\ldots,d\}^{n}$. Sums over $d$-indices are denote by $\sum_{I\in\mathcal{I}^{n}_{d}} = \sum_{i_{1}=1}^{d} \ldots \sum_{i_{n}=1}^{d}$ and partial derivatives are denoted in $d$-index notation by $\pd_{I} \varphi := \pd_{i_{1}} \ldots \pd_{i_{n}} \varphi$. Using this notation, the result of \cite{MR1227524} states that for any $d$-index $I\in\mathcal{I}^{n}_{d}$ of length $n\in\mathbb{N}$ it holds that
\begin{equation}
\label{pdIf}
(\pd_{I} \varphi)(x) = \sum_{j=0}^{\left\lfloor\frac{n}{2}\right\rfloor} p^{j}_{I}(x) (D^{n-j} f_{\varphi})(|x|) \,,
\end{equation}
where $p^{j}_{I}$ are homogeneous polynomial functions of degree $n-2j$ that are given by
\begin{equation}
\label{Polynomials}
p^{j}_{I}(x) = \frac{1}{2^{j}j!} \Delta_{x}^{j} \Big( x_{i_{1}} \ldots x_{i_{n}} \Big) \,.
\end{equation}
This yields \Cref{pdalphaphi}. Then, upon multiplying \Cref{pdIf} by $p^{i}_{I}(\tfrac{x}{|x|})$ and summing the result over all $I\in\mathcal{I}^{n}_{d}$, we infer the equations
\begin{equation}
\label{linearSystem}
\sum_{I\in\mathcal{I}^{n}_{d}} p^{i}_{I}(\tfrac{x}{|x|}) (\pd_{I} \varphi)(x) =
\sum_{j=0}^{\left\lfloor\frac{n}{2}\right\rfloor} \Big(
\sum_{I\in\mathcal{I}^{n}_{d}} p^{i}_{I}(\tfrac{x}{|x|}) p^{j}_{I}(\tfrac{x}{|x|})
\Big) |x|^{n-2j} (D^{n-j} f_{\varphi})(|x|)
\end{equation}
for $i=0,\ldots,\left\lfloor\frac{n}{2}\right\rfloor$. Since the Laplace operator is invariant under rotations $R\in SO(d)$, i.e.,
\begin{equation*}
(\Delta \varphi)\circ R = \Delta (\varphi\circ R) \,,
\end{equation*}
we infer from formula \eqref{Polynomials}
\begin{equation*}
p^{n}_{I}(Rx) = \sum_{J\in\mathcal{I}^{n}_{d}} R_{i_{1}j_{1}} \ldots R_{i_{n} j_{n}} p^{n}_{J}(x)
\end{equation*}
for all $R \in SO(d)$ with components $R_{ij}$, $i,j=1,\ldots,d$. Now, for each fixed $x\in\RR^{d}\setminus\{0\}$ there is a rotation $R = R(x)\in SO(d)$ such that $\tfrac{x}{|x|} = R e_{d}$, where $e_{d}\in \RR^{d}$ is the $d$-th standard basis vector. Consequently, together with
\begin{equation*}
\sum_{\ell = 1}^{d} R_{\ell i} R_{\ell j} = \delta_{ij} \,,
\end{equation*}
we get that the symmetric coefficients
\begin{equation}
\label{GramMatrix}
\sum_{I\in\mathcal{I}^{n}_{d}}
p^{i}_{I}(\tfrac{x}{|x|})
p^{j}_{I}(\tfrac{x}{|x|}) =
\sum_{I\in\mathcal{I}^{n}_{d}}
p^{i}_{I}(e_{d})
p^{j}_{I}(e_{d}) =: \gamma_{ij}(d,n)
\end{equation}
which appear for $i,j=0,\ldots,\left\lfloor\frac{n}{2}\right\rfloor$ in \Cref{linearSystem}, are the same for all $x\in\overline{\BB^{d}_{r}}$ and hence \Cref{linearSystem} can be read as a linear system of equations
\begin{equation}
\label{linearSystem2}
\sum_{I\in\mathcal{I}^{n}_{d}} p^{i}_{I}(\tfrac{x}{|x|}) (\pd_{I} \varphi)(x) =
\sum_{j=0}^{\left\lfloor\frac{n}{2}\right\rfloor} \gamma_{ij}(d,n) |x|^{n-2j} (D^{n-j} f_{\varphi})(|x|)
\end{equation}
for $i=0,\ldots,\left\lfloor\frac{n}{2}\right\rfloor$. To prove \Cref{Dnf}, we show that this linear system is invertible. For this, observe that $\gamma_{ij}(d,n)$ are the entries of the Gram matrix of the vectors
\begin{equation}
\label{vectors}
v_{i} = \Big( p^{i}_{I}(e_{d}) \Big)_{I\in\mathcal{I}^{n}_{d}} \in \RR^{d^{n}} \,, \qquad i=0,\ldots,\left\lfloor\frac{n}{2}\right\rfloor \,,
\end{equation}
and that this Gram matrix is invertible if and only if the vectors \eqref{vectors} are linearly independent. To prove the latter, let $\lambda_{0},\ldots,\lambda_{\left\lfloor\frac{n}{2}\right\rfloor}\in\RR$ such that $\lambda_{0} v_{0} + \ldots + \lambda_{\left\lfloor\frac{n}{2}\right\rfloor} v_{\left\lfloor\frac{n}{2}\right\rfloor} = 0$, i.e.,
\begin{equation*}
\sum_{i=0}^{\left\lfloor\frac{n}{2}\right\rfloor} \lambda_{i} p^{i}_{I}(e_{d}) = 0 \quad\text{for all } I \in \mathcal{I}^{n}_{d} \,.
\end{equation*}
First, consider in this linear combination the entry
\begin{equation*}
I=(i_{1},\ldots,i_{n}) \qquad\text{determined by}\qquad
i_{\ell} =
\renewcommand{\arraystretch}{1.5}
\left\{
\begin{array}{llc}
d & \text{if} &1 \leq \ell \leq 2\left\lceil\frac{n}{2}\right\rceil-n \,, \\
1 & \text{if} &2\left\lceil\frac{n}{2}\right\rceil-n < \ell \leq n \,.
\end{array}
\right.
\end{equation*}
For this particular $d$-index, formula \eqref{Polynomials} implies
\begin{equation*}
p^{j}_{I}(e_{d})
\neq 0 \quad \text{if } j = \left\lfloor\frac{n}{2}\right\rfloor \,,
\qquad\text{but}\qquad
p^{j}_{I}(e_{d}) = 0 \quad \text{if } 0 \leq j < \left\lfloor\frac{n}{2}\right\rfloor \,.
\end{equation*}
This forces $\lambda_{\left\lfloor\frac{n}{2}\right\rfloor} = 0$. To proceed, assume that $\lambda_{\left\lfloor\frac{n}{2}\right\rfloor},\ldots,\lambda_{\left\lfloor\frac{n}{2}\right\rfloor-i} = 0$ for some $0\leq i < \left\lfloor\frac{n}{2}\right\rfloor$. Then choose the entry
\begin{equation*}
I=(i_{1},\ldots,i_{n}) \qquad\text{determined by}\qquad
i_{\ell} =
\renewcommand{\arraystretch}{1.5}
\left\{
\begin{array}{llc}
d & \text{if} & 1 \leq \ell \leq 2\Big(\left\lceil\frac{n}{2}\right\rceil + i + 1 \Big) - n \,, \\
1 & \text{if} & 2\Big(\left\lceil\frac{n}{2}\right\rceil + i + 1 \Big) - n < \ell \leq n \,,
\end{array}
\right.
\end{equation*}
so that we have once more by formula \eqref{Polynomials}
\begin{equation*}
p^{j}_{I}(e_{d})
\neq 0 \quad \text{if } j = \left\lfloor\frac{n}{2}\right\rfloor - i - 1 \,,
\qquad\text{but}\qquad
p^{j}_{I}(e_{d}) = 0 \quad \text{if } 0 \leq j < \left\lfloor\frac{n}{2}\right\rfloor - i - 1 \,.
\end{equation*}
Hence, $\lambda_{\left\lfloor\frac{n}{2}\right\rfloor-i-1} = 0$ follows. From this procedure we infer $\lambda_{\left\lfloor\frac{n}{2}\right\rfloor} = \ldots = \lambda_{0} = 0$ and therefore the Gram matrix composed of the entries $\gamma_{ij}(d,n)$ is invertible. Denoting by $\gamma^{ij}(d,n)$ the components of the inverse of the Gram matrix, we can invert the linear system of equations \eqref{linearSystem2} from above for each $x\in\overline{\BB^{d}_{r}}$ to get
\begin{equation*}
|x|^{n} (D^{n} f_{\varphi})(|x|) = \sum_{I\in\mathcal{I}^{n}_{d}} \sum_{j=0}^{\left\lfloor\frac{n}{2}\right\rfloor}  \gamma^{0j}(d,n) p^{j}_{I}(\tfrac{x}{|x|}) (\pd_{I} \varphi)(x) \,.
\end{equation*}
Rearranging the right-hand side in terms of multi-indices yields \Cref{Dnf}.
\end{proof}
\begin{remark}
The identities \eqref{pdalphaphi}, \eqref{Dnf} in \Cref{DerRad} are pointwise and therefore hold for functions $\varphi\in C^{\infty}_{\mathrm{rad}}(\RR^{d})$ with radial representative $f\in C^{\infty}_{\mathrm{ev}}([0,\infty))$ too.
\end{remark}
\section{Hardy-type inequalities}
In this section, we prove Hardy-type inequalities in order to trade singular weights for derivatives within $L^{p}$-norms.
\begin{lemma}
\label{Hardy}
Let $p,r,s\in\RR$ with $p\geq 1$, $r>0$ and $s > -\frac{1}{p}$.
\begin{enumerate}[itemsep=1em,topsep=1em]
\item There is a constant $C_{p,r,s} > 0$ such that the estimate
\begin{equation*}
\big\| (\,.\,)^{s} f \big\|_{L^{p}(0,r)} \leq C_{p,r,s} \Big( |f(r)| + \big\| (\,.\,)^{s+1} f' \big\|_{L^{p}(0,r)} \Big)
\end{equation*}
holds for all $f\in C^1([0,r])$.
\medskip
\item There is a constant $C_{p,s} > 0$ such that the estimate
\begin{equation*}
\big\| (\,.\,)^{s} f \big\|_{L^{p}(0,\infty)} \leq C_{p,s} \big\| (\,.\,)^{s+1} f' \big\|_{L^{p}(0,\infty)}
\end{equation*}
holds for all $f\in C^{1}_{\mathrm{c}}([0,\infty))$.
\end{enumerate}
\end{lemma}
\begin{proof}
We have by the product rule
\begin{equation*}
x^{ps} |f(x)|^{p} = \pd_{x} \Big( \frac{x^{ps+1}}{ps+1} |f(x)|^{p} \Big) - \frac{x^{ps+1}}{ps+1} \pd_{x} |f(x)|^{p} 
\end{equation*}
for each $x\in[0,r]$ with $f(x)\neq 0$. We get for the latter term
\begin{align*}
- \frac{x^{ps+1}}{ps+1} \pd_{x} |f(x)|^{p} &= - \frac{x^{ps+1}}{ps+1} \Re\Big( -p |f(x)|^{p-2} f(x) \overline{f'(x)} \Big) \\&=
\Re\Big( - x^{(p-1)s} |f(x)|^{p-2} f(x) \overline{\frac{p}{ps+1} x^{s+1} f'(x)} \Big) \\&\leq
\Big( 1-\frac{1}{p} \Big) x^{ps}|f(x)|^{p} +
\frac{1}{p} \Big( \frac{p}{ps+1} \Big)^{p} x^{p(s+1)} |f'(x)|^{p}
\end{align*}
for each $x\in[0,r]$ with $f(x)\neq 0$. Rearranging this implies the differential inequality
\begin{equation}
\label{DifferentialInequality}
x^{ps} |f(x)|^{p} \leq \pd_{x} \Big( \frac{p}{ps+1} x^{ps+1} |f(x)|^{p} \Big) + \Big( \frac{p}{ps+1} \Big)^{p} x^{p(s+1)} |f'(x)|^{p} \,.
\end{equation}
Integration of inequality \eqref{DifferentialInequality} over the respective interval yields since $s>-\frac{1}{p}$
\begin{equation*}
\int_{0}^{r} x^{ps} |f(x)|^{p} \dd x \leq \frac{p}{ps+1} r^{ps+1} |f(r)|^{p} + \Big( \frac{p}{ps+1} \Big)^{p} \int_{0}^{r} x^{p(s+1)} |f'(x)|^{p} \dd x
\end{equation*}
which gives in both cases the desired inequality of norms.
\end{proof}
The boundary term in \Cref{Hardy} is controlled by the following inequality.
\begin{lemma}
\label{HardyBoundary}
Let $p,r,s\in\RR$ with $p\geq 1$, $r>0$ and $s > -\frac{1}{p}$. There is a constant $C_{p,r,s} > 0$ such that the estimate
\begin{equation*}
|f(r)| \leq C_{p,r,s} \Big( \big\| (\,.\,)^{s} f \big\|_{L^{p}(0,r)} + \big\| (\,.\,)^{s+1} f' \big\|_{L^{p}(0,r)} \Big)
\end{equation*}
holds for all $f\in C^1([0,r])$.
\end{lemma}
\begin{proof}
The fundamental theorem of calculus and Jensen's integral inequality imply
\begin{align*}
|f(r)|^{p} &= \Big| \frac{1}{r^{s+1}} \int_{0}^{r} \pd_{x} \Big( x^{s+1} f(x) \Big) \dd x \Big|^{p} \\&\leq
\frac{1}{r^{ps+1}} \int_{0}^{r} \Big| (s+1) x^{s} f(x) + x^{s+1} f'(x) \Big|^{p} \dd x \\&\leq
\frac{2^{p-1}(s+1)^{p}}{r^{ps+1}} \int_{0}^{r} x^{ps} |f(x)|^{p} \dd x + \frac{2^{p-1}}{r^{ps+1}} \int_{0}^{r} x^{p(s+1)} |f'(x)|^{p} \dd x \,,
\end{align*}
which yields the desired inequality.
\end{proof}
\section{Proof of the main results}
We combine the previous lemmas to prove \Cref{THMnorm} first.
\begin{proof}[Proof of \Cref{THMnorm}.] ``$\lesssim$'': For any $\varphi\in C^{\infty}_{\mathrm{rad}}(\overline{\BB^{d}_{r}})$, let us denote by $f_{\varphi}\in C^{\infty}_{\mathrm{ev}}([0,r])$ the radial representative with $\varphi(x) = f_{\varphi}(|x|)$ according to \Cref{RepresentativeRadial}. We note with a transformation of the integral that
\begin{equation*}
\| \varphi \|_{L^{p}(\BB^{d}_{r})} \simeq \Big\| (\,.\,)^{\frac{d-1}{p}} f_{\varphi} \Big\|_{L^{p}(0,r)}
\end{equation*}
for all $\varphi\in C^{\infty}_{\mathrm{rad}}(\overline{\BB^{d}_{r}})$. Now, fix $k\in\NN_{0}$. Let $1\leq n \leq k$ and $\alpha\in\mathbb{N}_{0}^{d}$ be a multi-index of length $|\alpha| = n$. The first part of \Cref{DerRad} implies with repeated applications of \Cref{Hardy,HardyBoundary}
\begin{align*}
\| \pd^{\alpha} \varphi \|_{L^{p}(\BB^{d}_{r})} &\lesssim \sum_{j=0}^{\left\lfloor\frac{n}{2}\right\rfloor} \Big\| (\,.\,)^{\frac{d-1}{p} + n - 2j} D^{n-j} f_{\varphi} \Big\|_{L^{p}(0,r)} \\&\lesssim
\sum_{j=1}^{\left\lfloor\frac{n}{2}\right\rfloor} |(D^{n-j} f_{\varphi})(r)| + \Big\| (\,.\,)^{\frac{d-1}{p} + n} D^{n} f_{\varphi} \Big\|_{L^{p}(0,r)} \\&\lesssim
\sum_{j=\left\lceil\frac{n}{2}\right\rceil}^{n} \Big\| (\,.\,)^{\frac{d-1}{p} + j} D^{j} f_{\varphi} \Big\|_{L^{p}(0,r)}
\end{align*}
for all $\varphi\in C^{\infty}_{\mathrm{rad}}(\overline{\BB^{d}_{r}})$. Summation over all multi-indices $0\leq |\alpha| \leq k$ yields the bound
\begin{equation}
\label{OneSide}
\| \varphi \|_{W^{k,p}(\BB^{d}_{r})} \lesssim \sum_{j=0}^k \Big\| (\,.\,)^{\frac{d-1}{p}+j} D^{j} f_{\varphi} \Big\|_{L^{p}(0,r)}
\end{equation}
for all $\varphi\in C^{\infty}_{\mathrm{rad}}(\overline{\BB^{d}_{r}})$. On $\RR^{d}$, one uses the second part of \Cref{Hardy} instead.
\par\medskip
``$\gtrsim$'': Conversely, the second part of \Cref{DerRad} implies the pointwise estimate
\begin{equation*}
\Big| |x|^{n} (D^{n} f_{\varphi})(|x|) \Big| \leq
\sum_{|\alpha|=n} \Big| q_{\alpha}(\tfrac{x}{|x|}) \pd_{x}^{\alpha} \varphi(x) \Big| \lesssim
\Big( \sum_{|\alpha|=n} |\pd^{\alpha}_{x} \varphi(x)|^{p} \Big)^{\frac{1}{p}}
\end{equation*}
for all $\varphi\in C^{\infty}_{\mathrm{rad}}(\overline{\BB^{d}_{r}})$ and all $x\in\overline{\BB^{d}_{r}}$. This yields the other inequality
\begin{equation}
\label{OtherSide}
\sum_{j=0}^k \Big\| (\,.\,)^{\frac{d-1}{p}+j} D^{j} f_{\varphi} \Big\|_{L^{p}(0,r)} \lesssim \| \varphi \|_{W^{k,p}(\BB^{d}_{r})}
\end{equation}
for all $\varphi\in C^{\infty}_{\mathrm{rad}}(\overline{\BB^{d}_{r}})$. Finally, for any radial representative $f_{\varphi}\in C^{\infty}_{\mathrm{ev}}([0,r])$, let us denote by $\widetilde{f}_{\varphi} \in C^{\infty}([0,r^{2}])$ the smooth function such that $f_{\varphi}(\rho) = \widetilde{f}_{\varphi}(\rho^{2})$ according to \Cref{RepresentativeEven}. Then $(D^{j} f_{\varphi})(\rho) = 2^{j} \widetilde{f}_{\varphi}^{(j)}(\rho^{2})$ and we infer from \eqref{OneSide}, \eqref{OtherSide} the equivalence of norms
\begin{equation*}
\| \varphi \|_{W^{k,p}(\BB^{d}_{r})} \simeq \sum_{j=0}^k \Big\| (\,.\,)^{\frac{d-1}{p}+j} D^{j} f_{\varphi} \Big\|_{L^{p}(0,r)} \simeq \sum_{j=0}^k \Big\| (\,.\,)^{\frac{d-2}{2p}+\frac{j}{2}} \widetilde{f}_{\varphi}^{(j)} \Big\|_{L^{p}(0,r^{2})}
\end{equation*}
for all $\varphi\in C^{\infty}_{\mathrm{rad}}(\overline{\BB^{d}_{r}})$. The result for homogeneous Sobolev norms follows analogously.
\end{proof}
Next, we establish a complete characterization of the radial Sobolev spaces, using our equivalent radial Sobolev norms.
\begin{proof}[Proof of \Cref{THMspace}.]
Let $W^{k,p}\big((0,r^{2}),w\big)$ be the weighted Sobolev space as defined via \eqref{WeightedSobolevNorm}. The densely defined linear operators
\begin{equation*}
\tr: C^{\infty}_{\mathrm{rad}}(\overline{\BB^{d}_{r}}) \subset W^{k,p}_{\mathrm{rad}}(\BB^{d}_{r}) \rightarrow W^{k,p}\big( (0,r^{2}),w \big) \,, \qquad \varphi \mapsto \varphi(\sqrt{\,.\,}e_{d}) \,,
\end{equation*}
and
\begin{equation*}
\ext: C^{\infty}([0,r^{2}]) \subset
W^{k,p}\big( (0,r^{2}),w \big) \rightarrow W^{k,p}_{\mathrm{rad}}(\BB^{d}_{r}) \,, \qquad f \mapsto f(|\,.\,|^{2}) \,.
\end{equation*}
are well-defined and, according to \Cref{RepresentativeEven,RepresentativeRadial}, satisfy $\tr(\varphi) \in C^{\infty}([0,r^{2}])$ with $\ext(\tr(\varphi)) = \varphi$ for all $\varphi \in C^{\infty}_{\mathrm{rad}}(\overline{\BB^{d}_{r}})$ and $\ext(f) \in C^{\infty}_{\mathrm{rad}}(\overline{\BB^{d}_{r}})$ with $\tr(\ext(f)) = f$ for all $f\in C^{\infty}([0,r^{2}])$. It follows from \Cref{THMnorm} that the operators $\tr$ and $\ext$ are bounded and thus have unique bounded linear extensions to the closure of their respective domains which are inverses of each other. In the same way, an isomorphism is also established for homogeneous Sobolev norms on $\RR^{d}$.
\end{proof}
Together with \Cref{THMnorm} and the previous lemmas, we are able to determine Sobolev norms of corotational maps.
\begin{proof}[Proof of \Cref{THMcorot}.]
``$\lesssim$'': For any corotational map $F\in C^{\infty}(\overline{\BB^{d}_{r}},\CC^{d})$ let us denote by $f\in C^{\infty}_{\mathrm{ev}}([0,r])$ the radial profile such that $F_{i}(x) = x_{i} f(|x|)$ for $i=1,\ldots,d$. By induction, we get for all multi-indices $\alpha\in\mathbb{N}_{0}^{d}$ the identity
\begin{equation}
\label{pdalphacorot}
\pd^{\alpha}_{x} F_{i}(x) = x_{i} \pd^{\alpha}_{x} f(|x|) + \alpha_{i} \pd^{\alpha-e_{i}}_{x} f(|x|) \,.
\end{equation}
For any integer $0 \leq n \leq k$, this implies the estimate
\begin{equation*}
\sum_{i=1}^{d} \sum_{|\alpha|=n} \big| \pd^{\alpha}_{x} F_{i}(x) \big|^{2} \lesssim \sum_{|\alpha|=n} |x|^{2} \big| \pd^{\alpha}_{x} f(|x|) \big|^{2} + \sum_{|\beta|=n-1} \big| \pd^{\beta}_{x} f(|x|) \big|^{2}
\end{equation*}
for all corotational maps $F\in C^{\infty}(\overline{\BB^{d}_{r}},\CC^{d})$ and all $x\in\overline{\BB^{d}_{r}}$. Integration of this inequality over $\BB^{d}_{r}$ gives with the first part of \Cref{DerRad,Hardy,HardyBoundary} the estimate
\begin{align*}
\sum_{i=1}^{d} \sum_{|\alpha|=n} \big\| \pd^{\alpha} F_{i} \big\|_{L^{2}(\BB^{d}_{r})} &\lesssim
\sum_{|\alpha|=n} \big\| |\,.\,| \pd^{\alpha} f(|\,.\,|) \big\|_{L^{2}(\BB^{d}_{r})} + \sum_{|\beta|=n-1} \big\| \pd^{\beta} f(|\,.\,|) \big\|_{L^{2}(\BB^{d}_{r})} \\&\lesssim
\sum_{j=\left\lceil\frac{n-1}{2}\right\rceil}^{n} \Big\| (\,.\,)^{\frac{d+1}{2}+2j-n} D^{j} f \Big\|_{L^{2}(0,r)} \\&\lesssim
\sum_{j=\left\lceil\frac{n-1}{2}\right\rceil}^{n-1} \big| D^{j}f(r) \big| + \Big\| (\,.\,)^{\frac{d+1}{2}+n} D^{n} f \Big\|_{L^{2}(0,r)} \\&\lesssim
\sum_{j=\left\lceil\frac{n-1}{2}\right\rceil}^{n} \Big\| (\,.\,)^{\frac{d+1}{2}+j} D^{j} f \Big\|_{L^{2}(0,r)}
\end{align*}
for all corotational maps $F\in C^{\infty}(\overline{\BB^{d}_{r}},\CC^{d})$. Now, \Cref{THMnorm} yields
\begin{equation*}
\| F \|_{H^{k}(\BB^{d}_{r})} \lesssim
\sum_{j=0}^{k} \Big\| (\,.\,)^{\frac{d+1}{2}+j} D^{j} f \Big\|_{L^{2}(0,r)} \simeq \| f(|\,.\,|) \|_{H^{k}(\BB^{d+2}_{r})}
\end{equation*}
for all corotational maps $F\in C^{\infty}(\overline{\BB^{d}_{r}},\CC^{d})$.
\par\medskip
``$\gtrsim$'': For the converse inequality, we start off with \Cref{THMnorm} and use the second part of \Cref{DerRad} to get the estimate
\begin{align*}
\| f(|\,.\,|) \|_{H^{k}(\BB^{d+2}_{r})} &\simeq
\sum_{j=0}^{k} \Big\| (\,.\,)^{\frac{d+1}{2}+j} D^{j}f \Big\|_{L^{2}(0,r)} \\&\simeq
\sum_{j=0}^{k} \Big\| |\,.\,|^{j+1} (D^{j}f)(|\,.\,|) \Big\|_{L^{2}(\BB^{d}_{r})} \\&\lesssim
\sum_{j=0}^{k} \sum_{|\alpha|=j} \Big\| |\,.\,| \pd^{\alpha} f(|\,.\,|) \Big\|_{L^{2}(\BB^{d}_{r})}
\end{align*}
for all $f\in C^{\infty}_{\mathrm{ev}}([0,r])$. The square of the absolute value of \Cref{pdalphacorot} yields the pointwise identity
\begin{equation*}
\big| \pd^{\alpha}_{x} F_{i}(x) \big|^{2} = (x_{i})^{2} \big| \pd^{\alpha}_{x} f(|x|) \big|^{2} + \pd_{x_{i}} \big( \alpha_{i} x_{i} \big| \pd^{\alpha-e_{i}}_{x} f(|x|)\big|^{2} \big) + \alpha_{i}(\alpha_{i}-1) \big| \pd^{\alpha-e_{i}}_{x} f(|x|) \big|^{2} \,.
\end{equation*}
Summing this over all $i=1,\ldots,d$ and rearranging gives
\begin{align*}
|x|^{2} \big| \pd^{\alpha}_{x} f(|x|) \big|^{2} = \sum_{i=1}^{d} \big| \pd^{\alpha}_{x} F_{i}(x) \big|^{2}
&- \sum_{i=1}^{d} \pd_{x_{i}} \big( \alpha_{i} x_{i} \big| \pd^{\alpha-e_{i}}_{x} f(|x|)\big|^{2} \big) \\&-
\sum_{i=1}^{d} \alpha_{i}(\alpha_{i}-1) \big| \pd^{\alpha-e_{i}}_{x} f(|x|) \big|^{2} \,,
\end{align*}
so that integration of this equation over $\BB^{d}_{r}$ implies with the divergence theorem
\begin{equation*}
\Big\| |\,.\,| \pd^{\alpha} f(|\,.\,|) \Big\|_{L^{2}(\BB^{d}_{r})} \leq \sum_{i=1}^{d} \big\| \pd^{\alpha} F_{i} \big\|_{L^{2}(\BB^{d}_{r})} \,.
\end{equation*}
Now, the previous two estimates show
\begin{align*}
\| f(|\,.\,|) \|_{H^{k}(\BB^{d+2}_{r})} \lesssim \sum_{|\alpha|=n} \Big\| |\,.\,| \pd^{\alpha} f(|\,.\,|) \Big\|_{L^{2}(\BB^{d}_{r})} \lesssim
\sum_{i=1}^{d} \sum_{|\alpha|=n} \big\| \pd^{\alpha} F_{i} \big\|_{L^{2}(\BB^{d}_{r})} \simeq \| F \|_{H^{k}(\BB^{d}_{r})}
\end{align*}
for all corotational maps $F\in C^{\infty}(\overline{\BB^{d}_{r}},\CC^{d})$.
\end{proof}
\section*{Acknowledgments}
This work has been supported by the Vienna School of Mathematics (VSM). The author thanks Roland Donninger and Irfan Glogi\'{c} for helpful discussions and useful comments on a first draft of this paper.

\end{document}